\documentclass[11pt]{amsart}

\usepackage[margin=1.15in]{geometry}
\usepackage{amscd,amssymb, amsmath, setspace, wasysym}
\usepackage{graphicx}
\usepackage[all]{xy}

\usepackage{url}
\usepackage{hyperref}

\theoremstyle{plain}
\newtheorem{theorem}{Theorem}[section]
\newtheorem{prop}[theorem]{Proposition}

\newtheorem{claim}[theorem]{Claim}
\newtheorem{cor}[theorem]{Corollary} 
\newtheorem{conj}[theorem]{Conjecture}
\newtheorem{lemma}[theorem]{Lemma}

\theoremstyle{definition}

\newtheorem{rmk}[theorem]{Remark}

\newtheorem{ex}[theorem]{Example}
\newtheorem*{ex*}{Example}

\newcommand\sO{{\mathcal O}}

\newcommand\sH{{\mathcal H}}

\newcommand\sF{{\mathcal F}}
\newcommand\sG{{\mathcal G}}
\newcommand\sE{{\mathcal E}}

\newcommand{\codim}{{\rm codim}\,}
\newcommand{\ddim}{{\rm dim}\,}

\newcommand\rp{{\mathbf{P}}}
\newcommand\rz{{\mathbf{Z}}}

\newcommand\rd{{\mathbf{D}}}

\subjclass[2010]{14F05}
\title[Derived invariants and Hochschild homology]{Derived invariants of irregular varieties and Hochschild homology}

\author{Luigi Lombardi}
\address{Department of Mathematics, Statistics, and Computer Science\\
University of Illinois at Chicago, 851 S. Morgan Street, Chicago, IL, 60607}
\email{\url{lombardi@math.uic.edu}}

\begin{document}
\begin{abstract}
 We study the behavior of cohomological support loci of the canonical bundle
under derived equivalence of smooth projective varieties. 
This is achieved by investigating the derived invariance of a generalized version of Hochschild homology.
Furthermore, using techniques coming from birational geometry, we establish 
the derived invariance of the Albanese dimension for varieties having non-negative Kodaira dimension. We apply our 
machinery to study the derived invariance of the holomorphic Euler characteristic and of certain Hodge numbers for special classes
of varieties. Further applications concern the behavior of particular types of fibrations under derived equivalence.
\end{abstract}
\maketitle 

\section{Introduction}

It is now well known that derived equivalent varieties share quite a few invariants.
For instance: the dimension, the Kodaira dimension, the numerical dimension and the canonical ring are 
examples of derived invariants. 
In the paper \cite{PS}, by describing the 
behavior under derived equivalence of the Picard variety, 
Popa and Schnell establish 
the derived invariance of the number of linearly independent holomorphic one-forms.
In this paper we study the behavior under derived equivalence of other fundamental objects in the geometry of
\emph{irregular} varieties, i.e. those with positive \emph{irregularity} $q(X):=h^0(X,\Omega_X^1)$,
such as the cohomological support loci and the Albanese dimension. 
Applications of our techniques concern
the derived invariance of the holomorphic Euler characteristic of varieties with large Albanese dimension and the derived invariance
of some of the Hodge numbers of fourfolds again with large Albanese dimension. 
A further application concerns the behavior of 
fibrations of derived equivalent threefolds onto irregular varieties.
This work is motivated by 
a well-known conjecture predicting the derived invariance of all Hodge numbers 
and by a conjecture of Popa (see Conjectures \ref{intrP} and \ref{intrPV} and \cite{Po}).

The main tool we use to approach the problems described above is the comparison
of the cohomology groups of twists by topologically trivial line bundles of the canonical bundles of the varieties in play. 
This is achieved by studying a generalized version of Hochschild homology which takes into account an important isomorphism due to
Rouquier related to derived autoequivalences (see \cite{Rou} Th\'{e}or\`{e}me 4.18).
In this way we obtain a theoretical result of independent interest in the study of 
derived equivalences of smooth projective varieties, which we now present. 
To begin with, we recall the \emph{Hochschild cohomology and homology} of a smooth projective variety $X$:
\begin{equation*}\label{eqHH}
HH^*(X):=\bigoplus_k {\rm Ext}_{X\times X}^k\big(i_*\sO_X,i_*\sO_X\big),\quad 
HH_*(X):=\bigoplus_k {\rm Ext}_{X\times X}^k\big(i_*\sO_X,i_*\omega_X\big)
\end{equation*}
where $i:X\hookrightarrow X\times X$ is the 
diagonal embedding of $X$. The space $HH^*(X)$ has a structure of ring under composition of morphisms and $HH_*(X)$ is a graded
$HH^*(X)$-module with the same operation. 
Results of C\u{a}ld\u{a}raru and Orlov show that the Hochschild cohomology and homology are derived invariants 
(see \cite{Cal} Theorem 8.1 and \cite{Or} Theorem 2.1.8).
 More precisely, if $\Phi: \rd(X)\longrightarrow \rd(Y)$ is an equivalence of derived categories of smooth projective
 varieties, then it induces an isomorphism of rings $HH^*(X)\cong HH^*(Y)$ and an isomorphism of graded modules 
 $HH_*(X)\cong HH_*(Y)$ compatible with the isomorphism $HH^*(X)\cong HH^*(Y)$.
We now present the generalization of Hochschild homology mentioned above.
For a triple $(\varphi,L,m)\in {\rm Aut}^0(X)\times {\rm Pic}^0(X)\times \rz$, we define the graded $HH^*(X)$-module
\begin{equation*}\label{intrtwist}
HH_*(X,\varphi,L,m):=\bigoplus_k {\rm Ext}_{X\times X}^k\big(i_*\sO_X,(1,\varphi)_*(\omega_X^{\otimes m}\otimes L)\big)
\end{equation*}
with module structure given by composition of morphisms. We think of these spaces as a ``twisted'' version
of the Hochschild homology of $X$. 
Lastly, we recall that a derived equivalence $\rd(X)\cong \rd(Y)$ induces an isomorphism of algebraic groups,
called \emph{Rouquier's isomorphism} 
\begin{equation}\label{intrrouq}
F:{\rm Aut}^0(X)\times {\rm Pic}^0(X)\longrightarrow {\rm Aut}^0(Y)\times {\rm Pic}^0(Y).
\end{equation}
(An explicit description of $F$ is given in \eqref{rouqkernel} (see \cite{Rou} Th\'{e}or\`{e}me 4.18,
\cite{Hu} Proposition 9.45 and \cite{Ros} Theorem 3.1; \emph{cf}. \cite{PS} footnote at p. 531).)
The following theorem describes the behavior of the twisted Hochschild homology under derived equivalence.
Its proof follows the general strategy of the proofs of Orlov and C\u{a}ld\u{a}raru, 
but further technicalities appear due to the possible presence of non-trivial automorphisms of $X$ and $Y$; 
see \S 2 for its proof.
\begin{theorem}\label{intrHH}
 Let $\Phi:\rd(X)\longrightarrow \rd(Y)$ be an equivalence of derived categories of smooth projective varieties 
defined over an algebraically closed field and let $m\in \rz$.
If $F(\varphi,L)=(\psi,M)$ (where $F$ is the Rouquier isomorphism), then $\Phi$ induces an isomorphism of graded modules
\begin{gather*}
HH_*(X,\varphi,L,m)\cong HH_*(Y,\psi,M,m)
\end{gather*}
compatible with the isomorphism $HH^*(X)\cong HH^*(Y)$.
\end{theorem}

We now move our attention to the main application of Theorem \ref{intrHH}, namely the behavior of
\emph{cohomological support loci} under derived equivalence. 
These loci are defined as
$$V^k(\omega_X):=\{L\in {\rm Pic}^0(X)\;|\; h^k(X,\omega_X\otimes L)>0\}$$
where $X$ is a smooth projective variety and $k\geq 0$ is an integer. 
From here on we work over the field of the complex numbers.
The $V^k(\omega_X)$'s have been studied for instance in \cite{GL1}, \cite{GL2}, \cite{EL}, \cite{A}, \cite{Ha}, \cite{PP2}.
They are one of the most important tools in the birational study of irregular varieties; 
roughly speaking, they control the geometry of the Albanese map and the fibrations 
onto lower dimensional irregular varieties.
The following conjecture, and its weaker variant, predicts the behavior of cohomological 
support loci under derived equivalence. 
As a matter of notation, 
we denote by $V^k(\omega_X)_0$ the union of the irreducible components of $V^k(\omega_X)$ passing through the origin.
\begin{conj}[\cite{Po} Conjecture 1.2]\label{intrP}
If $X$ and $Y$ are smooth projective derived equivalent varieties, then  
$$V^k(\omega_X)\cong V^k(\omega_Y)\quad \mbox{ for all }\quad k\geq 0.$$
\end{conj}
\begin{conj}[\cite{Po} Variant 1.3]\label{intrPV}
 Under the assumptions of Conjecture \ref{intrP}, there exist isomorphisms
$$V^k(\omega_X)_0\cong V^k(\omega_Y)_0\quad \mbox{ for all }\quad k\geq 0.$$
\end{conj}

It is important to emphasize that for all the applications we are interested in (e.g. invariance of the Albanese dimension, invariance of the
holomorphic Euler characteristic, invariance of Hodge numbers) it is in fact enough to verify Conjecture \ref{intrPV}. 
We also remark that Conjecture \ref{intrP} 
holds for varieties of general type since the cohomological support loci are birational 
invariants, while derived equivalent varieties of general type are birational by \cite{Ka2} Theorem 1.4.
Moreover, in \cite{Po} \S 2 it has been shown that Conjecture \ref{intrP} holds for surfaces as well.

In \S 3 we try to attack the above conjectures for varieties of arbitrary dimension.
To begin with, we show that Theorem \ref{intrHH} implies the derived invariance of 
$V^0(\omega_X)$ (see Proposition \ref{hzero}).
On the other hand, due to the possible presence of non-trivial automorphisms, 
the study of the derived invariance of the higher cohomological support loci
is more involved.
Nonetheless, by using a version of the Hochschild-Kostant-Rosenberg isomorphism and Brion's structural 
results on the actions of non-affine groups on smooth varieties,
we are able to show the derived invariance of $V^1(\omega_X)_0$ (see Corollary \ref{V1}).
The next theorem summarizes the main results on the derived invariance of these loci.
\begin{theorem}\label{intrV}
 Let $X$ and $Y$ be smooth projective derived equivalent varieties. Then the Rouquier isomorphism induces isomorphisms of algebraic sets
\begin{enumerate}
 \item $V^0(\omega_X)\cong V^0(\omega_Y)$.\\
\item $V^0(\omega_X)\cap V^1(\omega_X)\cong V^0(\omega_Y)\cap V^1(\omega_Y)$.\\
\item $V^1(\omega_X)_0\cong V^1(\omega_Y)_0$.\\
\end{enumerate}
\end{theorem}
We note that
(i) also holds if we consider arbitrary powers of the canonical bundle (see Proposition \ref{hzero}).
We point out also that cases in which the Rouquier isomorphism induces the full isomorphism 
$V^1(\omega_X)\cong V^1(\omega_Y)$ occur for instance when either $X$ is of maximal Albanese dimension (see Corollary \ref{v1max}), or when
the neutral component of the automorphism group, ${\rm Aut}^0(X)$, is affine (see Remark \ref{rmksplit});
Theorem \ref{intrV} is proved in \S 3.

Next we study Conjectures \ref{intrP} and \ref{intrPV} for varieties of dimension three.
In the process we recover Conjecture \ref{intrP} in dimension two as well making the isomorphisms on cohomological support loci 
explicit. 
In the following theorem we collect all results concerning the behavior of cohomological support loci of derived equivalent threefolds.
We denote by ${\rm alb}_X:X\longrightarrow {\rm Alb}(X)$ the Albanese map of $X$ and we say that $X$ is of \emph{maximal Albanese dimension}
if $\ddim {\rm alb}_X(X)=\ddim X$, i.e. ${\rm alb}_X$ is generically finite onto its image.
\begin{theorem}\label{intr3F}
 Let $X$ and $Y$ be smooth projective irregular derived equivalent threefolds. Then 
\begin{enumerate}
\item Conjecture \ref{intrPV} holds.\\
\item Conjecture \ref{intrP} holds if one of the following hypotheses is satisfied\\
\begin{enumerate}
\item $X$ is of maximal Albanese dimension.\\ 
\item $V^k(\omega_X)={\rm Pic}^0(X)$ for some $k\geq 0$ (for instance, by
\cite{PP2} Theorem E, $V^0(\omega_X)={\rm Pic}^0(X)$ 
whenever ${\rm alb}_X(X)$ is not fibered 
in sub-tori and $V^0(\omega_X)\neq \emptyset$).\\
\item ${\rm Aut}^0(X)$ is affine (for instance, by a theorem of Nishi, \cite{Ma} Theorem 2, 
this again happens when ${\rm alb}_X(X)$ is not fibered in sub-tori).\\
\end{enumerate}
\item If $q(X)\geq 2$, then $\ddim V^k(\omega_X)=\ddim V^k(\omega_Y)$ for all $k\geq 0$.
\end{enumerate}
\end{theorem}
%
Point (iii) brings evidence to a further variant of Conjecture \ref{intrP} predicting the invariance of the dimensions of cohomological
support loci (\emph{cf}. \cite{Po} Variant 1.4); partial results for the case $q(X)=1$ are described in Remark \ref{q=1}.
Since the proofs of Theorems \ref{intrV} and \ref{intr3F} extend to analogous results regarding cohomological support loci of 
bundles of holomorphic $p$-forms, when possible we will prove them in such generality.
Please refer to Theorem \ref{3-foldsorigin} and \S 6 for the proof of Theorem \ref{intr3F}.

Finally, we move our attention to applications of Theorems \ref{intrV}, \ref{intr3F} and \ref{intralb}.
The first regards the behavior of the Albanese dimension, $\ddim {\rm alb}_X(X)$, under derived equivalence.
According to Conjecture \ref{intrPV}, the Albanese dimension is 
expected to be preserved under derived equivalence as it can be read off 
from the dimensions of the $V^k(\omega_X)_0$'s (\emph{cf}. \eqref{foralb}). 
By Theorem \ref{intr3F}, this reasoning shows the derived invariance of the Albanese dimension for varieties of dimension up to three.
In higher dimension we establish this invariance for varieties having non-negative Kodaira dimension $\kappa(X)$
using the derived invariance of the irregularity and an extension of a result due to Chen-Hacon-Pardini 
(\cite{HP} Proposition 2.1, \cite{CH2} Corollary 3.6) on the geometry of the Albanese map via the Iitaka fibration; see \S 5.
\begin{theorem}\label{intralb}
 Let $X$ and $Y$ be smooth projective derived equivalent varieties. If $\ddim X\leq 3$, or if $\ddim X>3$ and $\kappa(X)\geq 0$, then
$$\ddim {\rm alb}_X(X)=\ddim {\rm alb}_Y(Y).$$ 
\end{theorem} 

The second application concerns the holomorphic Euler characteristic.
This is expected to be the same for arbitrary derived equivalent smooth projective varieties 
since the Hodge numbers are expected to be preserved (which is known to hold in dimension up to three; \emph{cf}. \cite{PS} 
Corollary C). We deduce this for varieties of large Albanese dimension as a consequence of the previous results and generic vanishing.
\begin{cor}\label{intreuler}
 Let $X$ and $Y$ be smooth projective derived equivalent varieties.
If $\ddim {\rm alb}_X(X)=\ddim X$, or if $\ddim {\rm alb}_X(X)=\ddim X-1$ and $\kappa(X)\geq 0$, then 
$$\chi(\omega_X)=\chi (\omega_Y).$$
\end{cor}
An immediate consequence is the derived invariance of two of the Hodge numbers for fourfolds satisfying the
hypotheses of Corollary \ref{intreuler}.
\begin{cor}\label{intrh02}
Let $X$ and $Y$ be smooth projective derived equivalent fourfolds. If $\ddim {\rm alb}_X(X)=4$, or if $\ddim {\rm alb_X}(X)=3$
and $\kappa(X)\geq 0$, then $$h^{0,2}(X)=h^{0,2}(Y)\quad \mbox{ and }\quad h^{1,3}(X)=h^{1,3}(Y).$$
\end{cor}
We remark that in \cite{PS} Corollary 3.4 the authors establish the invariance of $h^{0,2}$ and $h^{1,3}$ under different hypotheses, namely 
when ${\rm Aut}^0(X)$ is not affine (we recall that $h^{0,4}$, $h^{0,3}$, $h^{0,1}$ and $h^{1,2}$ are 
always known to be invariant, \emph{cf.} \cite{PS}). 
Corollaries \ref{intreuler} and \ref{intrh02} are proved in \S 7.

We now present our last application, in a direction which is one of the main motivations for Conjectures \ref{intrP} and \ref{intrPV} as 
explained in \cite{Po}.
From the classification of Fourier-Mukai equivalences for surfaces (\cite{Ka2}, \cite{BM}), it is known that if
$X$ admits a fibration 
$f:X\longrightarrow C$ onto a smooth curve of genus $\geq 2$, then any of its Fourier-Mukai partners admits 
a fibration onto the same curve.
Here we use our analysis, and 
a theorem of Green-Lazarsfeld regarding the properties of 
positive-dimensional irreducible components of the cohomological support loci, to investigate the behavior of fibrations of derived equivalent 
threefolds onto irregular varieties.
Recall that a smooth variety $X$ is called of \emph{Albanese general type} if ${\rm alb}_X$ 
is non-surjective and generically finite onto its image. 
The proof of the next corollary is contained in Proposition \ref{corfib} and Remark \ref{rmkfib}.
\begin{cor}\label{intrfibra}
Let $X$ and $Y$ be smooth projective derived equivalent threefolds.
There exists a morphism $f:X\longrightarrow W$ with connected fibers onto a normal variety $W$ of dimension $\leq 2$ such that any smooth 
model of $W$ is of Albanese general type if and only if $Y$ has a fibration of the same type.
Moreover, there exists a morphism $f:X\longrightarrow C$ with connected fibers onto a smooth curve of genus $\geq 2$ 
if and only if there exists a morphism $h:Y\longrightarrow D$ with connected fibers onto a smooth curve of genus $\geq 2$.\\
\end{cor}

To conclude we remark that while the approach in this paper relies in 
part on techniques of \cite{PS}, the key new ingredient is their interaction 
with the twisted Hochschild homology, introduced and studied here.
We are hopeful that this general method will find further applications in the future.
\section{Derived invariance of the twisted Hochschild homology}
We aim to prove Theorem \ref{intrHH}.
Its proof is based on a technical lemma extending previous computations carried out by
C\u{a}ld\u{a}raru and Orlov (\emph{cf}. \cite{Cal} Proposition 8.1 and \cite{Or} isomorphism (10)).
Let $X$ and $Y$ be smooth projective varieties defined over an algebraically closed field
and let $p$ and $q$ be the projections from $X\times Y$ onto the first and second factor respectively.
We denote by $\rd(X):=D^b(\mathcal{C}oh(X))$
the bounded derived category of coherent sheaves on $X$.  
We use the same symbol to denote a functor and its associated derived functor.
We say that $X$ and $Y$ are \emph{derived equivalent} if there exists an object
$\sE\in \rd(X\times Y)$
defining an equivalence of derived categories (i.e. an exact linear equivalence of triangulated categories)
$\Phi_{\sE}:\rd(X)\longrightarrow \rd(Y)$ as $\sF\mapsto q_*(p^*\sF\otimes \sE)$.
By \cite{Or} Proposition 2.1.7 and by denoting 
by $p_{rs}$ the projections from $X\times X\times Y\times Y$
onto the $(r,s)$-factor, an equivalence $\Phi_{\sE}:\rd(X)\longrightarrow \rd(Y)$ induces another equivalence
$$\Phi_{\sE^*\boxtimes \sE}:\rd(X\times X)\longrightarrow \rd(Y\times Y)$$
where $\sE^*\stackrel{\rm def}{=}\sH om(\sE,\sO_{X\times Y})\otimes p^*\omega_X[\ddim X]$
and  $$\sE^*\boxtimes \sE\stackrel{\rm def}{=}p^*_{13}\sE^* \otimes p^*_{24}\sE \in \rd(X\times X\times Y\times Y).$$
For automorphisms $\varphi\in {\rm Aut}^0(X)$ and $\psi\in {\rm Aut}^0(Y)$ we define the embeddings
$(1,\varphi):X\hookrightarrow X\times X$, $x\mapsto (x,\varphi(x))$ and $(1,\psi):Y\hookrightarrow Y\times Y$, $y\mapsto (y,\psi(y))$.
Lastly, we denote by $i$ and $j$ the diagonal embeddings of $X$ and $Y$ respectively.

\begin{lemma}\label{image}
Let $\Phi_{\sE}:\rd(X)\longrightarrow \rd(Y)$ be an equivalence of derived categories of smooth projective varieties and $F$ be the induced
Rouquier isomorphism. Let 
$m\in \rz$. If $F(\varphi,L)=(\psi,M)$,  
then $$\Phi_{\sE^*\boxtimes \sE}\big((1,\varphi)_*(\omega_X^{\otimes m} 
\otimes L)\big)\cong(1,\psi)_*(\omega_Y^{\otimes m}\otimes M).$$ 
\end{lemma}
\begin{proof}
We denote by $t_r$ the projection from $Y\times X\times Y$ onto the $r$-th factor and by $t_{rs}$ the projection
onto the $(r,s)$-th factor.
Consider the fiber product diagram

\centerline{ \xymatrix@=32pt{  
& Y\times X\times Y \ar[d]^{t_2}\ar[r]^{\lambda} & X\times X\times Y\times Y \ar[d]^{p_{12}}\\ 
& X\ar[r]^{(1,\varphi)} & X\times X}}
\noindent where $\lambda(y_1,x,y_2)=(x,\varphi(x),y_1,y_2)$.
By base change and the projection formula we get
\begin{eqnarray}\label{fm}
 \Phi_{\sE^*\boxtimes \sE}\big((1,\varphi)_*(\omega_X^{\otimes m}\otimes L)\big) & \cong & 
p_{34*}\big(p_{12}^*(1,\varphi)_*(\omega_X^{\otimes m}\otimes L)
\otimes \big(\sE^*\boxtimes \sE\big) \big)\\
& \cong &p_{34*}\big(\lambda_*t_2^*(\omega_X^{\otimes m}\otimes L)\otimes p_{13}^*\sE^*\otimes p_{24}^*\sE\big) \notag \\
& \cong &p_{34*}\lambda_*\big(t_2^*(\omega_X^{\otimes m}\otimes L)\otimes \lambda^*p_{13}^*\sE^*\otimes \lambda^*p_{24}^*\sE\big) \notag \\
& \cong &t_{13*}\big(t_2^*(\omega_X^{\otimes m}\otimes L)\otimes t_{21}^*\sE^*\otimes t_{23}^*(\varphi\times 1)^*\sE\big). \notag \\ \notag
\end{eqnarray}
By \cite{Or} p. 535, the equivalence $\Phi_{\sE}$ induces an isomorphism of objects 
$\sE\otimes p^*\omega_X\cong \sE\otimes q^*\omega_Y$ and, by \cite{PS} Lemma 3.1, isomorphisms
\begin{equation}\label{rouqkernel}
(\varphi\times 1)^*\sE\otimes p^*L\cong (1\times \psi)_*\sE\otimes q^*M\quad \mbox{whenever}\quad F(\varphi,L)=(\psi,M).
\end{equation}
Therefore
$p^*(\omega_X^{\otimes m}\otimes L)\otimes (\varphi\times 1)^*\sE\cong q^*(\omega_Y^{\otimes m}
\otimes M)\otimes (1\times \psi)_*\sE$. 
By pulling this isomorphism back via $t_{23}:Y\times X\times Y\longrightarrow X\times Y$, 
we have
\begin{equation}\label{fm2}
t_2^*(\omega_X^{\otimes m}\otimes L)\otimes t_{23}^*(\varphi\times 1)^*\sE\cong t_3^*(\omega_Y^{\otimes m}
\otimes M)\otimes t_{23}^*(1\times \psi)_*\sE.
\end{equation}
The morphism $t_3:Y\times X\times Y\longrightarrow Y$ can be rewritten 
as $t_3=\sigma_2\circ t_{13}$
where $\sigma_2:Y\times Y\longrightarrow Y$ is the projection onto the second factor. 
Let $\rho:Y\times X\longrightarrow X\times Y$ be
the inversion morphism $(y,x)\mapsto (x,y)$.
Then by \eqref{fm} and \eqref{fm2}
\begin{eqnarray*}
\Phi_{\sE^*\boxtimes \sE}\big((1,\varphi)_*(\omega_X^{\otimes m}\otimes L)\big) & \cong & 
t_{13*}\big(t_3^*(\omega_Y^{\otimes m}\otimes M)\otimes t_{21}^*\sE^*\otimes t_{23}^*(1\times \psi)_*\sE\big) \\
& \cong & t_{13*}\big(t_{13}^*\sigma_2^*(\omega_X^{\otimes m}\otimes L)\otimes t_{21}^*\sE^*\otimes t_{23}^*(1\times \psi)_*\sE\big) \\
& \cong & \sigma_2^*(\omega_Y^{\otimes m}\otimes M)\otimes t_{13*}\big(t_{21}^*\sE^*\otimes t_{23}^*(1\times \psi)_*\sE\big)\\
& \cong & \sigma_2^*(\omega_Y^{\otimes m}\otimes M)\otimes t_{13*}\big(t_{12}^*\rho^*\sE^*\otimes t_{23}^*(1\times \psi)_*\sE\big).
\end{eqnarray*}
We note that the object $t_{13*}\Big(t_{12}^*\rho^*\sE^*\otimes t_{23}^*(1\times \psi)_*\sE\Big)\in 
\rd(Y\times Y)$ is the 
kernel of the composition $\Phi_{(1\times \psi)_{*\sE}} \circ \Phi_{\rho^*\sE^*}$ (see  
\cite{Or} Proposition 2.1.2).
Furthermore, $\Phi_{(1\times \psi)_{*\sE}}\cong \psi_*\circ \Phi_{\sE}$ and 
$\Phi_{\rho^*\sE^*}\cong \Psi_{\sE^*}$ (where $\Psi_{\sE^*}:\rd (Y)\longrightarrow \rd(X)$
is defined as $\sG\mapsto p_*(q^*\sG\otimes \sE^*)$).
Hence 
$$\Phi_{(1\times \psi)_{*\sE}} \circ \Phi_{\rho^*\sE^*} \cong \psi_* \circ \Phi_{\sE}\circ \Psi_{\sE^*}\cong \psi_*\circ {\rm id}_{\rd(Y)}
\cong \psi_*,$$ which follows since $\Psi_{\sE^*}$ is the left (and right)
adjoint of $\Phi_{\sE}$.
On the other hand, the kernel of the derived functor $\psi_*:\rd(Y)\longrightarrow \rd(Y)$ 
is the structure sheaf of the graph of $\psi$, i.e. $\sO_{\Gamma_{\psi}}\cong(1,\psi)_*\sO_Y$ (see 
\cite{Hu} Example 5.4).
Thus, by the uniqueness of the Fourier-Mukai kernel
we have the isomorphism
$$ t_{13*}\big(t_{12}^*\rho^*\sE^*\otimes t_{23}^*(1\times \psi)_*\sE\big)\cong (1,\psi)_*\sO_Y.$$
To recap
\begin{eqnarray*}
 \Phi_{\sE^*\boxtimes \sE}\big(i_*(\omega_X^{\otimes m}\otimes L)\big) & \cong & 
 \sigma_2^*(\omega_Y^{\otimes m}\otimes M)\otimes (1,\psi)_*\sO_Y\\
& \cong & (1,\psi)_*\big( (1,\psi)^*\sigma_2^*(\omega_Y^{\otimes m} \otimes M)\otimes \sO_Y\big)\\
& \cong & (1,\psi)_*\big(\psi^*(\omega_Y^{\otimes m}\otimes M)\big)\\
& \cong & (1,\psi)_*(\omega_Y^{\otimes m}\otimes M).
\end{eqnarray*}
The last isomorphism follows as 
the action of ${\rm Aut}^0(X)$ on ${\rm Pic}^0(X)$ is trivial (\emph{cf}. \cite{PS} footnote at p. 531).
\end{proof}

\begin{proof}[Proof of Theorem \ref{intrHH}.]
By Lemma \ref{image}, the equivalence 
$\Phi_{\sE^*\boxtimes \sE}$ induces isomorphisms on the graded components of $HH_*(X,\varphi, L,m)$ and $HH_*(Y,\psi,M,m)$
\begin{eqnarray*}
{\rm Ext}^k_{X\times X}\big(i_*\sO_X,(1,\varphi)_*(\omega_X^{\otimes m}\otimes L)\big) & \cong & {\rm Ext}_{Y\times Y}^k
\big(\Phi_{\sE^*\boxtimes \sE}(i_*\sO_X),
\Phi_{\sE^*\boxtimes \sE}\big((1,\varphi)_*(\omega_X^{\otimes m}\otimes L)\big)\big)\\
& \cong & {\rm Ext}_{Y\times Y}^k\big(j_*\sO_Y,(1,\psi)_*(\omega^{\otimes m}_Y\otimes M)\big).
\end{eqnarray*}
Since $\Phi_{\sE^*\boxtimes \sE}$ commutes with the composition of any two morphisms, it follows that in particular 
$\Phi_{\sE^*\boxtimes \sE}$ induces an isomorphism of graded modules. 
\end{proof}

Theorem \ref{intrHH} will be often used in the following weaker form

\begin{cor}\label{formula}
Let $X$ and $Y$ be smooth projective derived equivalent varieties defined over an algebraically closed field of characteristic zero.
If $F(1,L)=(1,M)$, then
for any integers $m$ and $k\geq 0$ there exist isomorphisms
$$\bigoplus_{q=0}^k H^{k-q}\big(X,\Omega_X^{\ddim X-q}\otimes \omega_X^{\otimes m}\otimes L\big)\cong 
\bigoplus_{q=0}^k H^{k-q}\big(Y,\Omega_Y^{\ddim Y-q}\otimes \omega_Y^{\otimes m}\otimes M\big).$$
\end{cor}
\begin{proof}
 This follows immediately by Theorem \ref{intrHH} and by the general fact that the groups ${\rm Ext}^k_{X\times X}\big(i_*\sO_X,i_*\sF\big)$
decompose as $\bigoplus_{q=0}^k H^{k-q}\big(X,\Omega_X^{\ddim X-q}\otimes \omega_X^{-1}\otimes \sF\big)$ for any coherent sheaf $\sF$ and for all
$k\geq 0$
(see \cite {Ye} Corollary 4.7 and \cite{Sw} Corollary 2.6).
\end{proof}
%

\section{Behavior of cohomological support loci under derived equivalence}
In this section we study the behavior of cohomological support loci under derived equivalence.  
Applications of our analysis are given in \S 7. 
From now on we work over the field of the complex numbers.
\subsection{Cohomological support loci}
Let $X$ be a complex smooth projective irregular variety.
For a coherent sheaf $\sF$ on $X$ and integers $k\geq 0,r\geq 1$, we define the \emph{cohomological support loci of} $\sF$ as
$$V_r^k(\sF):=\{L \in {\rm Pic}^0(X)\; | \; h^k(X,\sF \otimes L)\geq r \}.$$ By semicontinuity 
these loci are algebraic closed sets in ${\rm Pic}^0(X)$. We set $V^k(\sF):=V^k_1(\sF)$ and 
we denote by $V^k_r(\sF)_0$ the union of the irreducible components of $V^k_r(\sF)$ passing through the origin.
By following the work of Pareschi and Popa \cite{PP2}, we say that
a coherent sheaf $\sF$ is a \emph{GV-sheaf} if 
$$\codim_{{\rm Pic}^0(X)} V^k(\sF)\geq k\quad \mbox{ for any }\quad k>0.$$
We denote by ${\rm alb}_X:X\longrightarrow {\rm Alb}(X)$ the Albanese map of $X$ and we recall the formula
(\emph{cf}. \cite{Po} p. 7)
\begin{equation}\label{foralb}
 \ddim {\rm alb}_X(X)={\rm min}\big\{\ddim X,{\rm min}_{k=0,\ldots ,\ddim X}\{\ddim X-k+\codim V^k(\omega_X)_0\}\big\}.
\end{equation}
Moreover, if $\ddim {\rm alb}_X(X)=\ddim X-k$, then there is a series of inclusions (\emph{cf}. 
\cite{PP2} Proposition 3.14 and \cite{GL1} Theorem 1; see also \cite{EL} Lemma 1.8)
\begin{equation}\label{forinc}
V^k(\omega_X)\supset V^{k+1}(\omega_X)\supset \ldots \supset V^d(\omega_X)=\{\sO_X\}.
\end{equation}

\subsection{Derived invariance of the zero-th cohomological support locus}
The following proposition proves and extends Theorem \ref{intrV} (i).
\begin{prop}\label{hzero}
Let $X$ and $Y$ be smooth projective derived equivalent varieties and let $F$ be the induced Rouquier isomorphism.
Let $m,r\in \rz$ with $r\geq 1$. If $L\in V^0_r(\omega_X^{\otimes m})$ and $F(1,L)=(\psi,M)$, then $\psi=1$ and $M\in V^0_r(\omega_Y^{\otimes m})$.
Moreover, $F$ induces an isomorphism of algebraic sets 
$$V^0_r(\omega_X^{\otimes m})\cong V^0_r(\omega_Y^{\otimes m}).$$ 
\end{prop}
\begin{proof}
Let $L\in V^0_r(\omega_X^{\otimes m})$.
By Theorem \ref{intrHH} and the adjunction formula we have
\begin{eqnarray*}
r\leq h^0(X,\omega_X^{\otimes m}\otimes L) & = & \ddim {\rm Hom}_{X\times X}\big(i_*\sO_X,i_*(\omega_X^{\otimes m}\otimes L)\big) \\ 
& = & \ddim {\rm Hom}_{Y\times Y}\big(j_*{\sO_Y},(1,\psi)_*(\omega_Y^{\otimes m}\otimes M)\big) \\
& = & \ddim {\rm Hom}_Y\big((1,\psi)^*j_*\sO_Y,\omega_Y^{\otimes m}\otimes M\big).
\end{eqnarray*}
Since $(1,\psi)^*j_*\sO_Y$ is supported on the locus of fixed points of $\psi$, we have $\psi=1$ and $M\in V^0_r(\omega_Y)$. 
Thus, $F$ maps $1\times V^0_r(\omega_X^{\otimes m})\mapsto 1\times V^0_r(\omega_Y^{\otimes m})$. In the same way
$F^{-1}$ maps $1\times V^0_r(\omega_Y^{\otimes m})\mapsto 1\times V^0_r(\omega_X^{\otimes m})$ yielding the wanted isomorphisms. 
%
%
\end{proof}

\subsection{Behavior of the higher cohomological support loci}
In this section we establish the isomorphism $V^1(\omega_X)_0\cong V^1(\omega_Y)_0$ of Theorem \ref{intrV}. 
It turns out that, by using the same techniques (i.e. invariance of the twisted Hochschild homology and Brion's results
on actions of non-affine groups), 
one can show a more general result which will be often used in the study of the invariance of 
cohomological support loci of bundles of holomorphic $p$-forms. 
\begin{theorem}\label{v1origin}
Let $X$ and $Y$ be smooth projective derived equivalent varieties
of dimension $d$ and let $F$ be the induced Rouquier isomorphism. Let $m\in \rz$. 
If $L\in \bigcup_{p,q\geq 0} V^p\big(\Omega_X^q\otimes\omega_X^{\otimes m}\big)_0$ and $F(1,L)=(\psi,M)$,  
then $\psi=1$ and $M\in \bigcup_{p,q\geq 0} V^p\big(\Omega_Y^q\otimes\omega_Y^{\otimes m}\big)_0$.
Moreover, $F$ induces isomorphisms of algebraic sets
$$\bigcup_{q=0}^k V^{k-q}\big(\Omega_X^{d-q}\otimes \omega_X^{\otimes m}\big)_0
\cong \bigcup_{q=0}^k V^{k-q}(\Omega_Y^{d-q}\otimes \omega_Y^{\otimes m}\big)_0\quad \mbox{ for any }\quad k\geq 0.$$
\end{theorem}
\begin{proof}
Before starting the proof, we recall some notation and facts from \cite{PS} Theorem A.
Let $\alpha:{\rm Pic}^0(Y)\longrightarrow {\rm Aut}^0(X)$ and $\beta:{\rm Pic}^0(X)\longrightarrow {\rm Aut}^0(Y)$ be morphisms defined
as
$$\alpha(M)={\rm pr}_1 (F^{-1}(1,M))\quad \mbox{ and }\quad \beta(L)={\rm pr}_1(F(1,L))$$ 
(here ${\rm pr}_1$ denotes the projection onto the first factor from the product ${\rm Aut}^0(\cdot)\times {\rm Pic}^0(\cdot)$).
We denote by $A$ and $B$ the images of $\alpha$ and $\beta$ respectively. We recall that $A$ and $B$ are isogenous abelian varieties.
  
If $A$ is trivial, then $F(1,{\rm Pic}^0(X))=(1,{\rm Pic}^0(Y))$. Hence, by Corollary \ref{formula}
$$F\big(1,\bigcup_q V^{k-q}\big(\Omega_X^{d-q}
\otimes \omega_X^{\otimes m}\big)\big)\subset 
\big(1,\bigcup_q V^{k-q}\big(\Omega_Y^{d-q}
\otimes \omega_Y^{\otimes m}\big)\big)\quad \mbox{ for any }\quad k\geq 0.$$
%
Since $B$ is trivial as well, we have 
$$F^{-1}\big(1, \bigcup_q V^{k-q}\big(\Omega_Y^{d-q}\otimes \omega_Y^{\otimes m}\big)\big)\subset \big(1, \bigcup_q V^{k-q}\big(\Omega_X^{d-q}
\otimes \omega_X^{\otimes m}\big)\big)\quad \mbox{ for any }\quad k\geq 0.$$ 

We suppose now that both $A$ and $B$ are non-trivial. We first show  
\begin{claim}\label{claimetale}
 For any $k\geq 0$, $F\big(1,\bigcup_q V^{k-q}\big(\Omega_X^{d-q}\otimes \omega_X^{\otimes m}\big)_0\big)\subset \big(1,{\rm Pic}^0(Y)\big)$.
\end{claim}
\begin{proof}
Since $A$ is positive-dimensional, the compact part of ${\rm Aut}^0(X)$ (i.e. ${\rm Alb}({\rm Aut}^0(X))$) 
is positive-dimensional as well. Thus Brion's results on actions of non-affine algebraic groups
imply that $X$ is an \'{e}tale locally trivial fibration $\xi:X\longrightarrow A/H$ 
where $H$ is a finite subgroup of $A$ (the proof of this fact is analogous to the one of \cite{PS} Lemma 2.4; see also \cite{Br}).
Let $Z$ be the smooth and connected fiber of $\xi$ over the origin of $A/H$.
Via base change we get a commutative diagram

\centerline{ \xymatrix@=32pt{  
& A\times Z\ar[d] \ar[r]^g & X\ar[d]^{\xi}\\ 
& A\ar[r] & A/H\\}} 
\noindent where $g(\varphi,z)=\varphi(z)$.
Let $(z_0,y_0)\in Z\times Y$ be an arbitrary point and let 
$$f=(f_1\times f_2):A\times B\longrightarrow X\times Y$$ be the orbit map $(\varphi,\psi)\mapsto (\varphi(z_0),\psi(y_0))$. 
The pull-back morphism $f^*=(f^*_1\times f_2^*):{\rm Pic}^0(X)\times {\rm Pic}^0(Y)\longrightarrow \widehat{A}\times 
\widehat{B}$ has finite kernel by a theorem of Nishi (see the first line of \cite{PS} at p. 533 and \cite{Ma} Theorem 2). 
In \cite{PS} p. 533 it is also shown that 
$$L\in ({\rm ker}\,f_1^*)_0\quad \Longrightarrow \quad F(1,L)=(1,M)\quad \mbox{ for some }\quad M\in {\rm Pic}^0(Y)$$
(here $({\rm ker}\,f_1^*)_0$ denotes the neutral component of ${\rm ker}\,f_1^*$).
So it is enough to show the inclusion
\begin{equation}\label{eqv}
\bigcup_q V^{k-q}\big(\Omega_X^{d-q}\otimes \omega_X^{\otimes m}\big)_0\subset ({\rm ker}\,f_1^*)_0\quad \mbox{ for any }\quad k\geq 0.
\end{equation}
This is achieved by computing cohomology spaces on $A\times Z$ via the \'{e}tale morphism 
$g$ and by using the fact that these computations are 
straightforward on $A$.
Let $p_1,p_2$ be the projections from the product $A\times Z$ onto the first and second factor respectively.
By denoting by $\nu:A\times \{z_0\}\hookrightarrow A\times Z$ the inclusion morphism, we have $g\circ \nu=f_1$. 
Moreover, via the isomorphism ${\rm Pic}^0(A\times Z)\cong {\rm Pic}^0(A)\times {\rm Pic}^0(Z)$, we obtain
$g^*L\cong p_1^*L_1\otimes p_2^*L_2$ where $L_1\in {\rm Pic}^0(A)$ and $L_2\in {\rm Pic}^0(Z)$.
Note also that $f_1^*L\cong \nu^*g^*L\cong L_1$.
Finally, if $L\in \bigcup_q V^{k-q}\big(\Omega_X^{d-q}\otimes \omega_X^{\otimes m}\big)$, then by 
\cite{La} Lemma 4.1.14
\begin{equation}\label{eq1}
0\neq \bigoplus_{q=0}^k H^{k-q}\big(X,\Omega_X^{d-q}\otimes \omega_X^{\otimes m}\otimes L\big)\subset 
\bigoplus_{q=0}^k H^{k-q}\big(A\times Z,\Omega_{A\times Z}^{d-q}\otimes \omega^{\otimes m}_{A\times Z}\otimes g^*L\big).
\end{equation}
By K\"{u}nneth's formula, the sum on the right hand side of \eqref{eq1} 
is non-zero only if $f_1^*L\cong \sO_A$, i.e. only if $L \in {\rm ker}\,f_1^*$. Therefore \eqref{eqv} follows.
\end{proof}

By Claim \ref{claimetale} and  
Corollary \ref{formula} we obtain that for any $k\geq 0$
the Rouquier isomorphism maps  
$$1\times \bigcup_q V^{k-q}\big(\Omega_X^{d-q}\otimes \omega_X^{\otimes m}\big)_0 \mapsto 
1\times \bigcup_q V^{k-q}\big(\Omega_Y^{d-q}\otimes \omega_Y^{\otimes m}\big)_0.$$
In complete analogy with Claim \ref{claimetale}, one can show 
$$M\in \bigcup_q V^{k-q}\big(\Omega_Y^{d-q}\otimes \omega_Y^{\otimes m}\big)_0 \quad \Longrightarrow \quad F^{-1}(1,M)=(1,L) 
 \mbox{ for some }  L\in {\rm Pic}^0(X).$$
So, by Corollary \ref{formula}, for any $k\geq 0$, $F^{-1}$ maps $$1\times \bigcup_q V^{k-q}\big(\Omega_Y^{d-q}\otimes \omega_Y^{\otimes m}\big)_0 
\mapsto 
1\times \bigcup_q V^{k-q}\big(\Omega_X^{d-q}\otimes \omega_X^{\otimes m}\big)_0.$$
\end{proof}
The following corollaries yield the proof of Theorem \ref{intrV} (iii) and (ii).
\begin{cor}\label{V1}
 Under the assumptions of Theorem \ref{v1origin}, the Rouquier isomorphism induces isomorphisms of algebraic sets
$$V^1_r(\omega_X)_0\cong V^1_r(\omega_Y)_0\quad \mbox{ for any }\quad r\geq 1.$$
\end{cor}
\begin{proof}
Let $L\in V^1_r(\omega_X)_0$. By Theorem \ref{v1origin}, 
$F(1,L)=(1,M)$ for some $M\in {\rm Pic}^0(Y)$. By Corollary \ref{formula} there is an isomorphism
\begin{equation*}\label{aux}
H^1(X,\omega_X\otimes L)\oplus H^0(X,\Omega_X^{d-1}\otimes L)\cong H^1(Y,\omega_Y\otimes M)\oplus H^0(Y,\Omega_Y^{d-1}\otimes M).
\end{equation*}
By Serre duality and the Hodge linear-conjugate isomorphism we have
$$h^0(X,\Omega_X^{d-1}\otimes L)=h^1(X,\omega_X\otimes L)\quad \mbox{and}\quad h^0(Y,\Omega_Y^{d-1}\otimes M)=h^1(Y,\omega_Y\otimes M).$$
Hence 
$h^1(X,\omega_X\otimes L)=h^1(Y,\omega_Y\otimes M)\geq r$ and $F$ induces then the wanted isomorphisms. 
\end{proof}

\begin{cor}\label{inter}
Under the assumptions of Theorem \ref{v1origin}, and for any integers $l,m,r,s$ with $r,s\geq 1$, the Rouquier isomorphism
induces isomorphisms of algebraic sets
\begin{gather*}
V^0_r(\omega_X^{\otimes m})\cap \Big(\bigcup_q V^{k-q}\big(\Omega_X^{d-q}\otimes \omega_X^{\otimes l}\big)\Big)\cong 
V^0_r(\omega_Y^{\otimes m})\cap \Big(\bigcup_q V^{k-q}\big(\Omega_Y^{d-q}\otimes \omega_Y^{\otimes l}\big)\Big)\\
V^0_r(\omega_X^{\otimes m})\cap V^1_s(\omega_X)\cong V^0_r(\omega_Y^{\otimes m})\cap V^1_s(\omega_Y).
\end{gather*}
\end{cor}
\begin{proof}
If $L\in V^0_r(\omega_X^{\otimes m})$, then $F(1,L)=(1,M)$ for some $M\in V^0_r(\omega_Y^{\otimes m})$ by Proposition 
\ref{hzero}. Then we conclude as in the proofs of Theorem \ref{v1origin} and Corollary \ref{V1}.
\end{proof}

 \begin{rmk}\label{rmksplit}
 It is important to note that, whenever $F(1,{\rm Pic}^0(X))=(1,{\rm Pic}^0(Y))$, the proofs 
 of Theorem \ref{v1origin} and Corollary \ref{V1} yield full isomorphisms
 $$\bigcup_q V^{k-q}\big(\Omega_X^{d-q}\otimes \omega_X^{\otimes m}\big)\cong \bigcup_q V^{k-q}\big(\Omega_Y^{d-q}
 \otimes \omega_Y^{\otimes m}\big)\quad \mbox{ for any }\quad k\geq 0$$ $$V^1_r(\omega_X)\cong V^1_r(\omega_Y).$$
 By Theorem \ref{v1origin}, this occurs either if $V^p(\Omega_X^q\otimes \omega_X^{\otimes m})={\rm Pic}^0(X)$ 
 for some $p,q\geq 0$ and $m\in \rz$,
 or if ${\rm Aut}^0(X)$ is affine (since in this case 
 the abelian variety $A$ in the proof of Theorem \ref{v1origin} is trivial).
 \end{rmk}

\section{Popa's conjectures in dimension two and three}
In this section we prove Theorem \ref{intr3F} (i). The proofs of (ii) and (iii) are postponed in \S 6 since they use
the derived invariance of the Albanese dimension discussed in \S 5. First of all 
we make a couple of considerations in the case of surfaces.
\subsection{The case of surfaces}
Conjecture \ref{intrP} has been proved by Popa in \cite{Po}. His proof is based on an 
explicit computation of cohomological support loci. 
By using Proposition \ref{hzero} and Corollary \ref{V1}, we can see more precisely that 
the isomorphisms $V^k(\omega_X)\cong V^k(\omega_Y)$ are induced by the Rouquier isomorphism.
Moreover, along the same lines we can also show that $F$ induces the further isomorphism $V^1(\Omega_X^1)\cong V^1(\Omega_Y^1)$
(\emph{cf}. \cite{Lo} for a detailed analysis).
\begin{ex}[Elliptic surfaces]\label{rmkgamma0}
Let $X$ be an elliptic surface of Kodaira dimension one and of maximal Albanese dimension
(i.e. an isotrivial elliptic surface fibered onto a curve of genus $\geq 2$). 
By following \cite{Be2}, we recall an invariant attached to this type of surfaces.
The surface $X$ admits a unique fibration $f:X\longrightarrow C$ onto a curve of genus $\geq 2$. 
We denote by $G$ the general fiber of $f$ and by ${\rm Pic}^0(X,f)$ the kernel of 
the pull-back $f^*$
$$0\longrightarrow {\rm Pic}^0(X,f)\longrightarrow {\rm Pic}^0(X)\stackrel{f^*}{\longrightarrow} {\rm Pic}^0(G).$$ 
In \cite{Be2} (1.6), it is proved that there exists a finite 
group $\Gamma^0(f)$ and an isomorphism $${\rm Pic}^0(X,f)\cong f^*{\rm Pic}^0(C)\times \Gamma^0(f).$$ 
The group $\Gamma^0(f)$ is the invariant mentioned above; it
is identified with the group of the connected components of ${\rm Pic}^0(X,f)$. 
Any Fourier-Mukai partner $Y$ of $X$ is an elliptic surface fibered over $C$ (\emph{cf}. \cite{BM} Proposition 4.4). We denote by
$g:Y\longrightarrow C$ this (unique) fibration and by $\Gamma^0(g)$ its invariant.
In \cite{Ph} Theorem 5.2.7, Pham proves that if $\rd(X)\cong \rd(Y)$ then
\begin{align}\label{eqgamma}
\Gamma^0(f)\cong \Gamma^0(g).
\end{align}
Here we note that \eqref{eqgamma} also follows by the derived invariance of the zero-th cohomological support locus. This goes as follows. 
By \cite{Be2} Corollaire 2.3 and the results in \cite{Po} p. 5, we get isomorphisms
$$V^0(\omega_X)=V^1(\omega_X)\cong {\rm Pic}^0(X,f)\cong f^*{\rm Pic}^0(C)\times \Gamma^0(f).$$
Since $Y$ is of maximal Albanese dimension too by the results in \S 4.1 and \eqref{foralb}, we have 
$$V^0(\omega_Y)= V^1(\omega_Y)\cong g^*{\rm Pic}^0(Y,g)\cong g^*{\rm Pic}^0(C)\times \Gamma^0(g).$$
Thus, by Proposition \ref{hzero} 
$$f^*{\rm Pic}^0(C)\times \Gamma^0(f)\cong V^0(\omega_X)\cong V^0(\omega_Y)\cong g^*{\rm Pic}^0(C)\times \Gamma^0(g),$$ which 
in particular yields \eqref{eqgamma}.
\end{ex}

\subsection{Proof of Theorem \ref{intr3F} (i)}

\begin{theorem}\label{3-foldsorigin}
Let $X$ and $Y$ be smooth projective derived equivalent threefolds. 
Then the Rouquier isomorphism induces isomorphisms 
$$V^p_r(\Omega_X^q)_0\cong V^p_r(\Omega_Y^q)_0\quad \mbox{ for any }\quad p,q\geq 0\quad\mbox{ and }\quad r\geq 1.$$
\end{theorem}
\begin{proof}
The isomorphisms $V^0_r(\omega_X)\cong V^0_r(\omega_Y)$ and $V^1_r(\omega_X)_0\cong V^1_r(\omega_Y)_0$ are proved in
Proposition \ref{hzero} and Corollary \ref{V1} respectively.
We now show $V^2_r(\omega_X)_0\cong V^2_r(\omega_Y)_0$.
By Theorem \ref{v1origin}, we have $F(1,V^2_r(\omega_X)_0)\subset (1,{\rm Pic}^0(Y))$ and hence, by Corollary \ref{formula},
$h^k(X,\omega_X\otimes L)=h^k(Y,\omega_Y\otimes M)$ whenever $F(1,L)=(1,M)$ and for $k=0,1$. 
By \cite{PS} Corollary C, the holomorphic Euler characteristic is a derived invariant in dimension three. This yields
$\chi(\omega_X\otimes L)=\chi(\omega_X)=\chi (\omega_Y)=\chi (\omega_Y\otimes M)$ and hence 
$h^2(X,\omega_X\otimes L)=h^2(Y,\omega_Y\otimes M)$.
Thus, 
if $L\in V^2_r(\omega_X)_0$, then $M\in V^2_r(\omega_Y)_0$ and consequently
$F$ induces isomorphisms $V^2_r(\omega_X)_0\cong V^2_r(\omega_Y)_0$. 
This in turn yields isomorphisms $V^0_r(\Omega_X^1)_0\cong V^0_r(\Omega_Y^1)_0$.

We now prove the isomorphisms $V^1_r(\Omega_X^q)_0\cong V^1_r(\Omega_Y^q)_0$ for $q=1,2$. 
By Theorem \ref{v1origin}, we have $F(1, V^1_r(\Omega_X^q)_0)\subset (1,{\rm Pic}^0(Y))$.
By Serre duality and the Hodge linear-conjugate isomorphism, we get $h^0(X,\Omega_X^1\otimes L)=
h^2(X,\omega_X \otimes L)$ and $h^0(Y,\Omega_Y^1\otimes M)=h^2(Y,\omega_Y\otimes M)$. Consequently, by
Corollary \ref{formula} (with $m=0$ and $k=2$), we obtain
$h^1(X,\Omega_X^2\otimes L)=h^1(Y,\Omega_Y^2\otimes M)$ and therefore  
$F$ induces isomorphisms $V^1_r(\Omega_X^2)_0\cong V^1_r(\Omega_Y^2)_0$ for all $r\geq 1$. 
The isomorphisms $V^1_r(\Omega_X^1)_0\cong V^1_r(\Omega_Y^1)_0$ follow again by  
Corollary \ref{formula} (with $m=0$ and $k=3$).
\end{proof}

\section{Behavior of the Albanese dimension under derived equivalence}
In this section we prove Theorem \ref{intralb}.
Our main tool is a generalization of a result due to Chen-Hacon-Pardini saying that 
if $f:X\longrightarrow Z$ is a non-singular representative of the 
Iitaka fibration of a smooth projective variety $X$ of maximal Albanese dimension, 
then $$q(X)-q(Z)=\ddim X-\ddim Z$$ (see \cite{HP} Proposition 2.1 and \cite{CH2} Corollary 3.6). 
We generalize this fact in two ways: 1) we consider all possible values of the Albanese dimension of $X$ and 2) 
we replace the Iitaka fibration of $X$ with a more general class of morphisms.
\begin{lemma}\label{hacon}
Let $X$ and $Z$ be smooth projective varieties and let $f:X\longrightarrow Z$ be a surjective morphism with connected fibers.
If the general fiber of $f$ is a smooth variety with surjective Albanese map, then
$$q(X)-q(Z)=\ddim {\rm alb}_X(X)-\ddim {\rm alb}_Z(Z).$$
\end{lemma}
\begin{proof}
Due to the functoriality of the Albanese map, there exists a commutative diagram

\centerline{ \xymatrix@=32pt{  
 & X\ar[d]^f \ar[r] ^{{\rm alb}_X} & {\rm Alb}(X)\ar[d]^{f_*}\ar[d] \\ 
& Z \ar[r]^{{\rm alb}_Z} & {\rm Alb}(Z). \\ }}
\noindent The induced morphism $f_*$ is surjective since $f$ is so. 
As in the proof of \cite{HP} Proposition 2.1, one can show that $f_*$ has connected fibers and 
that the image of a general fiber of $f$ via ${\rm alb}_X$ is a translate of ${\rm ker} \,f_*$.
This implies that ${\rm alb}_X(X)$ is fibered in tori of dimension $q(X)-q(Z)$ over ${\rm alb}_Z(Z)$.
By the theorem on the dimension of the fibers of a morphism we get the stated equality.

\end{proof}

\begin{proof}[Proof of Theorem \ref{intralb}.]
The theorem holds for $\ddim X\leq 3$ by \S 4 and \eqref{foralb}, so we suppose $\ddim X> 3$.

If $\kappa(X)=\kappa (Y)=0$ then the Albanese maps of $X$ and $Y$ are surjective by \cite{Ka1} Theorem 1.
Thus the Albanese dimensions of $X$ and $Y$ are $q(X)$ and $q(Y)$ respectively which are equal by \cite{PS} Corollary B. 

We now suppose $\kappa(X)=\kappa(Y)>0$. 
Since the problem is invariant under birational modification, with a little abuse of notation, we consider
non-singular representatives $f:X\longrightarrow Z$ and $g:Y\longrightarrow W$ of the Iitaka fibrations of $X$ and $Y$ respectively 
(\emph{cf}. \cite{Mo} (1.10)). As the canonical rings of $X$ and $Y$ are isomorphic (\cite{Or} Corollary 2.1.9), 
it turns out that $Z$ and $W$ are birational varieties (\emph{cf}. \cite{Mo} Proposition 1.4 or \cite{To} p. 13).
By \cite{Ka1} Theorem 1, the morphisms $f$ and $g$ satisfy the hypotheses of Lemma \ref{hacon} which yields
$$q(X)-\ddim {\rm alb}_X(X)=q(Z)-\ddim {\rm alb}_Z(Z)=q(W)-\ddim {\rm alb}_W(W)=q(Y)-\ddim {\rm alb}_Y(Y).$$ We conclude 
as $q(X)=q(Y)$.
\end{proof}
The following corollary will play a central role in the proof of Theorem \ref{intr3F} (ii).
\begin{cor}\label{v1max}
 Let $X$ and $Y$ be smooth projective derived equivalent varieties with $X$ of maximal Albanese dimension and $F$ be the induced
Rouquier isomorphism.
If $F(1,L)=(\psi,M)$ with $L\in V^1_r(\omega_X)$, then $\psi=1$ and $M\in V^1_r(\omega_Y)$. Moreover, $F$ induces isomorphisms of 
algebraic sets 
$$V_r^1(\omega_X)\cong V_r^1(\omega_Y) \quad \mbox{ for any }\quad r\geq 1.$$ 
\end{cor}
\begin{proof}
By Theorem \ref{intralb}, $Y$ is of maximal Albanese dimension as well.
By \eqref{forinc}, we get two inclusions: $V^1_r(\omega_X)\subset V^0(\omega_X)$ and  
$V^1_r(\omega_Y)\subset V^0(\omega_Y)$.
We conclude then by applying Corollary \ref{inter}.
\end{proof}

\section{End of the proof of Theorem \ref{intr3F}}\label{seccsl}
\subsection{Proof of Theorem \ref{intr3F} (ii)}
The following two propositions imply Theorem \ref{intr3F} (ii). 
\begin{prop}\label{prop3Fsplit}
 Let $X$ and $Y$ be smooth projective derived equivalent threefolds. Assume 
that either ${\rm Aut}^0(X)$ is affine, or
$V^p(\Omega_X^q\otimes \omega_X^{\otimes m})={\rm Pic}^0(X)$ for some $p,q\geq 0$ and $m\in \rz$.
Then the Rouquier isomorphism induces isomorphisms of algebraic sets
$$V^p_r(\Omega_X^q)\cong V^p_r(\Omega^q_Y)\quad \mbox{ for all }\quad p,q\geq 0 \quad \mbox{ and }\quad r\geq 1.$$
\end{prop}
\begin{proof}
By Remark \ref{rmksplit} we have $F(1,{\rm Pic}^0(X))=(1,{\rm Pic}^0(Y))$.
The isomorphisms $V^0_r(\omega_X)\cong V^0_r(\omega_Y)$ and $V^1_r(\omega_X)\cong V^1_r(\omega_Y)$ hold 
by Proposition \ref{hzero} and Remark \ref{rmksplit} respectively. The isomorphisms $V^2_r(\omega_X)\cong V^2_r(\omega_Y)$ follow since 
$\chi (\omega_X)=\chi (\omega_Y)$ (\emph{cf}. \cite{PS} Corollary C). 
We now establish the isomorphisms $V^1_r(\Omega_X^2)\cong V^1_r(\Omega_Y^2)$. Let $L\in V^1_r(\Omega_X^2)$ and $F(1,L)=(1,M)$.
By Corollary \ref{formula} (with $m=0$ and $k=2$), Serre duality and the Hodge linear-conjugate isomorphism, we get
$h^1(X,\Omega_X^2\otimes L)=h^1(Y,\Omega_Y^2\otimes M)$. This says that $F$ maps $1\times V^1_r(\Omega_X^2)\mapsto 1\times 
V^1_r(\Omega_Y^2)$ inducing the 
wanted isomorphisms. 
For the isomorphisms $V^1_r(\Omega_X^1)\cong V^1_r(\Omega_Y^1)$, it is enough to argue as before by using Corollary \ref{formula} 
(with $m=0$ and $k=3$).

\end{proof}

\begin{prop}\label{propalbsplit3f}
 Let $X$ and $Y$ be smooth projective derived equivalent threefolds with $X$ of maximal Albanese dimension.
Then the Rouquier isomorphism induces isomorphisms of algebraic sets 
$$V^k_r(\omega_X)\cong V^k_r(\omega_Y)\quad \mbox{ for all }\quad k\geq 0\quad \mbox{ and }\quad r\geq 1.$$
\end{prop}
\begin{proof}
Proposition \ref{hzero} and Corollary \ref{v1max} yield
the isomorphisms
$V^k_r(\omega_X)\cong V^k_r(\omega_Y)$ for any $k\neq 2$, so we only focus on the case $k=2$.
Since by \eqref{forinc} $V_r^2(\omega_X)\subset V^0(\omega_X)$, we have that
$F(1,V^2_r(\omega_X))\subset (1,{\rm Pic}^0(Y))$ (see Proposition \ref{hzero}).
Hence, 
by Corollary \ref{formula}, we get that 
$h^k(X,\omega_X\otimes L)=h^k(Y,\omega_Y\otimes M)$ whenever $F(1,L)=(1,M)$ with $L\in V_r^2(\omega_X)$ and for $k=0,1$.
As $\chi(\omega_X)=\chi(\omega_Y)$, we get that $h^2(X,\omega_X\otimes L)=h^2(Y,\omega_Y\otimes M)$ and therefore $F$ and $F^{-1}$ induce 
%
the desired isomorphisms.
\end{proof}

\subsection{Proof of Theorem \ref{intr3F} (iii)}
The proof of Theorem \ref{intr3F} (iii) relies on the determination of the dimensions of cohomological support loci
for special types of threefolds
according to the invariants $\kappa (X)$, $q(X)$ and $\ddim {\rm alb}_X(X)$ which are preserved by derived equivalence. 
This study is of independent interest
and it does not use any results from the theory of the derived categories but rather it 
follows by generic vanishing theory, Koll\'{a}r's result on the degeneration of the Leray spectral sequence and classification theory
of algebraic surfaces.   

More specifically, thanks to Theorem \ref{intr3F} (ii), we can assume that $X$
is a threefold not of general type, not 
of maximal Albanese dimension and such that ${\rm Aut}^0(X)$ is not affine 
(therefore $\chi(\omega_X)=0$ by \cite{PS} Corollary 2.6).
Furthermore, by the derived invariance of $\chi(\omega_X)$ and by Proposition \ref{hzero} and Corollary \ref{inter}
we can assume $V^0(\omega_X)\subsetneq {\rm Pic}^0(X)$. 
Thanks to our previous results (Proposition \ref{hzero} and Theorem \ref{intralb}) 
and to the ones in \cite{PS}, the Fourier-Mukai partner $Y$ satisfies the same hypotheses as $X$.
At this point Theorem \ref{intr3F} (iii) follows since Propositions \ref{pr1} - \ref{pr5} classify $\ddim V^k(\omega_X)$ and 
$\ddim V^k(\omega_Y)$ in terms of derived invariants. The following two lemmas will be useful for our analysis.

\begin{lemma}\label{p2}
Let $X$ and $Y$ be smooth projective varieties and $f:X\longrightarrow Y$ a surjective morphism with connected fibers. 
If $h$ denotes the dimension of the general fiber of $f$, then $$f^*V^k(\omega_Y)\subset V^{k+h}(\omega_X)\quad\mbox{ for any }
\quad k=0,\ldots ,\ddim Y.$$
\end{lemma}
\begin{proof}
This follows by the degeneration of the Leray spectral sequence (\cite{Ko2} Theorem 3.1) 
$$E_2^{p,q}=H^p(Y,R^q f_*\omega_X\otimes L)\Longrightarrow H^{p+q}(X,\omega_X\otimes f^*L)\quad \mbox{ for any }
\quad L \in {\rm Pic}^0(Y)$$ and by
the fact $R^h f_*\omega_X\cong \omega_Y$ (\cite{Ko1} Proposition 7.6).
\end{proof}

\begin{lemma}\label{v0neg}
 Let $X$ be a smooth projective variety with $\kappa(X)=-\infty$. Then $V^0(\omega_X^{\otimes m})=\emptyset$ for any $m>0$.
\end{lemma}
\begin{proof}
 Suppose that $L\in V^0(\omega_X^{\otimes m})$ for some $m>0$. 
By \cite{CH2} Theorem 3.2, we can assume that $L$ is a torsion line bundle. Let $e$ be the order of $L$. If $H^0(X,
\omega_X^{\otimes m}\otimes L)\neq 0$ then it is easy to see that $H^0(X,\omega_X^{\otimes me})\neq 0$; this yields a contradiction.
%
%
\end{proof}
\begin{prop}\label{pr1}
 Let $X$ be a smooth projective threefold. Suppose $\kappa (X)=2$, $\ddim {\rm alb}_X(X)=2$, $\chi (\omega_X)=0$ and 
$V^0(\omega_X)\varsubsetneq {\rm Pic}^0(X)$. 
If $q(X)=2$, then $\ddim V^1(\omega_X)=1$ and $\ddim V^2(\omega_X)=0$. If $q(X)>2$, then  
 $\ddim V^1(\omega_X)=\ddim V^2(\omega_X)=q(X)-1$.
\end{prop}
\begin{proof}
Since the problem is invariant under birational modification, with a little abuse of notation, 
we consider a non-singular representative $f:X\longrightarrow S$ of the 
Iitaka fibration of $X$ (\emph{cf}. \cite{Mo} (1.10)). We study the commutative diagram

\centerline{ \xymatrix@=32pt{  
 & X  \ar[d]^f \ar[r]^{{\rm alb}_X} & {\rm Alb}(X)\ar[d]^{f_*}\\ 
& S\ar[r]^{{\rm alb}_S} & {\rm Alb}(S).\\ }}
\noindent
We first show
\begin{claim}\label{c3}
$\ddim {\rm alb}_S(S)=1$.
\end{claim}
\begin{proof}
If by contradiction $\ddim {\rm alb}_S(S)=0$, then by Lemma \ref{hacon} $q(S)=0$ and $q(X)=2$. This yields an absurd
since via ${\rm alb}_X$ a general fiber of $f$ is mapped onto a translate of ${\rm ker}\, f_*$.
Suppose now, again by contradiction, that $\ddim {\rm alb}_S(S)=2$. Then $S$ is of maximal Albanese dimension and
$q(X)=q(S)$. The morphism $f$ has connected fibers and induces an inclusion 
$f^*{\rm Pic}^0(S)\hookrightarrow {\rm Pic}^0(X)$ 
which is an isomorphism for dimension reasons. 
Since $S$ is a surface of general type, by Castelnuovo's Theorem we have 
$\chi (\omega_S)>0$ (\emph{cf}. \cite{Be} Theorem X.4) and hence $V^0(\omega_S)={\rm Pic}^0(S)$.
By applying Lemma \ref{p2} to $f$, we obtain 
$f^*V^0(\omega_S)\subset V^1(\omega_X)$ and hence $V^1(\omega_X)={\rm Pic}^0(X)$.
Since $V^0(\omega_X)\subsetneq {\rm Pic}^0(X)$ and $\chi(\omega_X)=0$, we obtain
$V^2(\omega_X)={\rm Pic}^0(X)$ which contradicts \eqref{foralb}.
\end{proof}

By the previous claim and Lemma \ref{hacon}, we have $q(X)=q(S)+1$. Moreover,   
${\rm alb}_S$ 
has connected fibers and the image of ${\rm alb}_S$ is a smooth curve of genus $q(S)$ (see \cite{Be} Proposition V.15). 
Via pull-back we get an injective homomorphism 
$f^*{\rm Pic}^0(S)\hookrightarrow {\rm Pic}^0(X)$. We distinguish two subcases: $q(S)=1$ and $q(S)>1$. 

We begin with the case $q(S)=1$.
Then $q(X)=2$ and both ${\rm alb}_X$ and ${\rm alb}_S$ are surjective.
Since $\chi(\omega_S)>0$, we have $V^0(\omega_S)={\rm Pic}^0(S)$ and, by 
Lemma \ref{p2}
$$f^*V^0(\omega_S)=f^*{\rm Pic}^0(S)\subset V^1(\omega_X).$$
Thus either $V^1(\omega_X)={\rm Pic}^0(X)$ or $\ddim V^1(\omega_X)=1$.
But, if $V^1(\omega_X)={\rm Pic}^0(X)$, then $V^2(\omega_X)={\rm Pic}^0(X)$ as $\chi (\omega_X)=0$.
This contradicts \eqref{foralb}. 
We have then 
\begin{equation}\label{eqv1}
\ddim V^1(\omega_X)=1.
\end{equation}
We now show that $$\ddim V^2(\omega_X)=0$$ (again in the case $q(S)=1$). 
Let $X\stackrel{a}{\longrightarrow} Z\stackrel{b}{\longrightarrow} {\rm Alb}(X)$ be the Stein factorization of ${\rm alb}_X$. 
The surface $Z$ is normal, the morphism $a$ has connected fibers and $b$ is finite. 
Let $g:X'\longrightarrow Z'$ be a non-singular representative of $a$, which still has connected fibers.
To compute $V^2(\omega_X)$, it is enough to compute $V^2(\omega_{X'})$ since $X$ and $X'$ are birational.
By studying the Leray spectral sequence
of $g$, $V^2(\omega_{X'})$ is in turn 
computed by looking at the $V^k(\omega_{Z'})$'s.
The following claim computes the cohomological support loci of $Z'$.
\begin{claim}
 $Z'$ is birational to an abelian surface.
\end{claim}
\begin{proof}
It is not hard to check that $Z'$ is of maximal Albanese dimension and hence that $\kappa(Z')\geq 0$.
If by contradiction $k(Z')=2$, then Castelnuovo's Theorem (see Claim \ref{c3}) would imply $V^0(\omega_Z')= {\rm Pic}^0(Z')$.
Thus, by Lemma \ref{p2}, we would get $$V^1(\omega_{X'})\supset g^*V^0(\omega_{Z'})=g^*{\rm Pic}^0(Z')$$ contradicting 
\eqref{eqv1}. 
If $k(Z')=1$, then $Z'$ would be birational to an elliptic surface of maximal Albanese dimension fibered onto a curve $B$ of genus $g(B)\geq 2$. 
This yields a contradiction as Lemma \ref{hacon} would imply $g(B)=1$. 
Hence $\kappa (Z')=0$ and $Z'$ is birational to an abelian surface as it is of maximal Albanese dimension.
\end{proof}

We now compute $V^2(\omega_{X'})$. Since $q(X')=q(Z')=2$, the morphism $g$ induces an isomorphism ${\rm Pic}^0(Z')\cong {\rm Pic}^0(X')$.
By the degeneration of the Leray spectral sequence
$$H^p(Z',R^qg_*\omega_{X'}\otimes L)\Longrightarrow H^{p+q}(X',\omega_{X'}\otimes g^*L)\quad \mbox{for any}\quad 
L\in {\rm Pic}^0(Z'),$$ 
and by using that $R^1g_*\omega_{X'}\cong \omega_{Z'}$ (\cite{Ko1} Proposition 7.6), 
we obtain
\begin{equation*}\label{eqspes}
H^2(X',\omega_{X'}\otimes g^*L)\cong H^1(Z',\omega_{Z'} \otimes L)\oplus H^2(Z',g_*\omega_{X'}\otimes L).
\end{equation*}
By \cite{PP2} Theorem 5.8, $g_*\omega_{X'}$ is a $GV$-sheaf on $Z'$. Hence 
$\codim_{{\rm Pic}^0(Z')} V^2(g_*\omega_{X'})\geq 2$ and consequently
$\codim _{{\rm Pic}^0(X)} V^2(\omega_X)=\codim_{{\rm Pic}^0(X')} V^2(\omega_{X'})\geq 2$.

We now study the case $q(S)=q(X)-1\geq 2$. We set $C:={\rm alb}_S(S)$. By \cite{Be} Lemma 
V.16, we have $$g(C)=q(S)\geq 2\quad \mbox{ and }\quad {\rm Pic}^0(S)={\rm alb}_S^*J(C).$$
By applying Lemma \ref{p2} twice, first to ${\rm alb}_S$ and then to $f$, we have
$$f^*{\rm alb}_S^*J(C)=f^*{\rm alb}_S^*V^0(\omega_C)\subset 
f^*V^1(\omega_S)\subset V^2(\omega_X)\subset {\rm Pic}^0(X).$$ 
This implies $\ddim V^2(\omega_X)\geq q(S)=q(X)-1$.
Since $V^2(\omega_X)\neq {\rm Pic}^0(X)$, we have that in fact $$\ddim V^2(\omega_X)=
q(S)=q(X)-1.$$ Moreover, by \eqref{forinc}, $V^1(\omega_X)\supset V^2(\omega_X)$ and thus
$\ddim V^1(\omega_X)=q(X)-1$ since $V^0(\omega_X)$ is a proper subvariety and $\chi (\omega_X)=0$. 
\end{proof}

\begin{prop}\label{pr2}
Let $X$ be a smooth projective threefold with $\kappa (X)=2$, $\ddim {\rm alb}_X(X)=1$, $\chi (\omega_X)=0$ and 
$V^0(\omega_X)\subsetneq
{\rm Pic}^0(X)$. If $q(X)=1$, then $\ddim V^1(\omega_X)\leq 0$ and $\ddim V^2(\omega_X)=0$.
If $q(X)>1$, then $V^1(\omega_X)=V^2(\omega_X)={\rm Pic}^0(X)$.
\end{prop}
\begin{proof}
%
As in the previous proof we denote by $f:X\longrightarrow S$ a non-singular representative of the Iitaka fibration of $X$. 
By Lemma \ref{hacon}, we have $q(X)-q(S)=1-\ddim {\rm alb}_S(S)$. We distinguish two cases: $\ddim {\rm alb}_S(S)=0$ 
and $\ddim {\rm alb}_S(S)=1$. 

If $\ddim {\rm alb}_S(S)=0$, then $q(S)=0$, $q(X)=1$ and ${\rm alb}_X$ is surjective
since $\ddim {\rm alb}_X(X)=1$. Moreover, ${\rm alb}_X$ has connected fibers by \cite{Ue} Lemma 2.11.
 Set $E:={\rm Alb}(X)$ and $a:={\rm alb}_X$.
  By \cite{Ko1} Proposition 7.6, $R^2a_*\omega_X\cong \sO_{E}$.
The degeneration of the Leray spectral sequence associated to $a$ yields decompositions
$$H^2(X,\omega_X\otimes a^*L)\cong H^1(E,R^1a_*\omega_X\otimes L)\oplus H^0(E,L) \mbox{ for any }
L\in {\rm Pic}^0(E)\cong{\rm Pic}^0(X).$$
By \cite{Ha} Corollary 4.2, $R^1a_*\omega_X$ is a $GV$-sheaf on $E$. Hence $\ddim V^2(\omega_X)=0$ and
$V^1(\omega_X)$ is either empty or zero-dimensional as $V^0(\omega_X)\subsetneq {\rm Pic}^0(X)$ and $\chi (\omega_X)=0$.

We now suppose $\ddim {\rm alb}_S(S)=1$. Then $q(X)=q(S)$ by Lemma \ref{hacon}. We distinguish two subcases: $q(S)=1$ and $q(S)>1$.
If $q(S)=q(X)=1$ then the image of ${\rm alb}_X$ is an elliptic curve and the same argument as in the previous case 
applies. Suppose now $q(S)=q(X)>1$. Then ${\rm alb}_S$ has connected fibers and its image is a smooth curve $B$ of genus $g(B)=q(S)>1$. Therefore 
$V^0(\omega_B)={\rm Pic}^0(B)$ and by Lemma \ref{p2} 
$${\rm alb}_S^*{\rm Pic}^0(B)={\rm alb}_S^*V^0(\omega_B)\subset V^1(\omega_S)\subset {\rm Pic}^0(S),$$ which forces 
$V^1(\omega_S)={\rm Pic}^0(S)$. Another application of Lemma \ref{p2} gives
$$f^*V^1(\omega_S)\subset V^2(\omega_X)\subset {\rm Pic}^0(X).$$
Therefore $V^2(\omega_X)={\rm Pic}^0(X)$ and consequently $V^1(\omega_X)={\rm Pic}^0(X)$.
\end{proof}

\begin{prop}\label{pr3}
 Let $X$ be a smooth projective threefold with $\kappa (X)=1$, $\chi (\omega_X)=0$ and $V^0(\omega_X)\subsetneq {\rm Pic}^0(X)$.
If $\ddim {\rm alb}_X(X)=2$, then $q(X)\geq 3$ and $\ddim V^1(\omega_X)=\ddim V^2(\omega_X)=q(X)-1$.
If $\ddim {\rm alb}_X(X)=1$, then $q(X)\geq 2$ and $V^1(\omega_X)=V^2(\omega_X)={\rm Pic}^0(X)$.
\end{prop}
\begin{proof}
This proof is completely analogous to the proofs of Propositions \ref{pr1} and \ref{pr2}.

%
%
\end{proof}

\begin{prop}\label{pr4}
 Let $X$ be a smooth projective threefold with $\kappa(X)=0$ and $\chi (\omega_X)=0$.
If $\ddim {\rm alb}_X(X)=2$, then $\ddim V^1(\omega_X)=\ddim V^2(\omega_X)=0$.
If $\ddim {\rm alb}_X(X)=1$, then $\ddim V^1(\omega_X)\leq 0$ and $\ddim V^2(\omega_X)=0$.
\end{prop}

\begin{proof}
We recall that, by \cite{CH1} Lemma 3.1, $V^0(\omega_X)$ consists of at most one point. 
We begin with the case $\ddim {\rm alb}_X(X)=2$.
By \cite{Ka1} Theorem 1, ${\rm alb}_X$ is surjective and has connected fibers. Therefore $q(X)=h^2(X,\omega_X)=2$ and 
$\{\sO_X\}\in V^1(\omega_X)$ since $\chi (\omega_X)=0$. 
 Set $a:={\rm alb}_X$.  
By \cite{Ha} Corollary 4.2
$$\codim V^1(a_*\omega_X)\geq 1 \quad \mbox{ and }\quad \codim V^2(a_*\omega_X)\geq 2.$$ Using that
 $R^1 a_*\omega_X\cong \sO_{{\rm Alb}(X)}$ (\cite{Ko1} Proposition 7.6) and by studying the Leray spectral sequence associated to $a$
we see that $$\codim V^1(\omega_X)\geq 1\quad \mbox{ and } \quad \codim V^2(\omega_X)\geq 2.$$
 At this point, the hypothesis $\chi (\omega_X)=0$ implies $\ddim V^1(\omega_X)=0$.

If $\ddim {\rm alb}_X(X)=1$, then as in the previous case $\ddim V^2(\omega_X)=0$ and consequently 
$V^1(\omega_X)$ is either empty or zero-dimensional since $\chi(\omega_X)=0$.  
\end{proof}

\begin{prop}\label{pr5}
 Let $X$ be a smooth projective threefold with $\kappa(X)=-\infty$ and $\chi (\omega_X)=0$.
\begin{enumerate}
 \item Suppose $\ddim {\rm alb}_X(X)=2$. If $q(X)=2$, then $V^1(\omega_X)=V^2(\omega_X)=\{\sO_X\}$. If $q(X)>2$, then 
$\ddim V^1(\omega_X)=\ddim V^2(\omega_X)=q(X)-1$.\\
\item Suppose $\ddim {\rm alb}_X(X)=1$. If $q(X)=1$, then $\ddim V^1(\omega_X)\leq 0$ and $\ddim V^2(\omega_X)=0$.
If $q(X)>1$, then $V^1(\omega_X)=V^2(\omega_X)={\rm Pic}^0(X)$.
\end{enumerate}
\end{prop}
\begin{proof}
By Lemma \ref{v0neg}, $V^0(\omega_X)=\emptyset$.
We start with the case $\ddim {\rm alb}_X(X)=2$.
Let $a:X\longrightarrow S\subset {\rm Alb}(X)$ be the Albanese map of $X$ and
$b:X\longrightarrow S'$ be the Stein factorization of $a$.
The morphism $b$ has connected fibers and $S'$ is a normal surface. 
Let $c:X'\longrightarrow S''$ be a non-singular representative of $b$. 
The morphism $c$ has connected fibers and its general fiber is isomorphic to $\rp^1$.
Hence, by Lemma \ref{hacon} $q(X')=q(S'')$.
Thanks to the degeneration of the Leray spectral sequence associated to $c$, we can compute the $V^k(\omega_{X'})$'s (and therefore the 
$V^k(\omega_X)$'s) by studying the $V^k(\omega_{S''})$'s.
We divide this study according to the values of $\kappa (S'')$. 
First of all, we note that $S''$ is of maximal Albanese dimension and hence that $\kappa (S'')\geq 0$.
Moreover, with an analogous argument as in Claim \ref{c3}, one can show that $\kappa (S'')<2$.
If $\kappa (S'')=0$, then $S''$ is birational to an abelian surface. This forces
$$q(X)=q(X')=q(S'')=2\quad \mbox{ and }\quad {\rm Pic}^0(X')\cong {\rm Pic}^0(S'').$$
We have $c_*\omega_{X'}=0$ as the general fiber of $c$ is $\rp^1$.  
By the degeneration of the Leray spectral sequence associated to $c$
we get $V^1(\omega_{X'})=V^2(\omega_{X'})=\{\sO_{X'}\}$.

If $\kappa(S'')=1$, then $S''$ is birational to an elliptic surface of maximal Albanese dimension fibered onto a curve of genus $\geq 2$.
Thus $\codim_ {{\rm Pic}^0(S'')} V^1(\omega_{S''})=1$ (see Example \ref{rmkgamma0}). 
Moreover $q(X')=q(S'')\geq 3$ and $c^*{\rm Pic}^0(S'')={\rm Pic}^0(X')$.
Another application of the Leray spectral sequence associated to $c$ yields 
$$\codim V^2(\omega_{X'})=1$$
and consequently 
$$V^1(\omega_{X'})=1$$ since $\chi (\omega_{X'})=0$ and 
$V^0(\omega_{X'})=\emptyset$.

We now suppose $\ddim {\rm alb}_X(X)=1$. Let $a:X\longrightarrow C\subset {\rm Alb}(X)$ be the Albanese map of $X$ where $C:={\rm Im}\, a$.
Then $a$ has connected fibers and $q(X)=g(C)$ by \cite{Ue} Lemma 2.11. We note that a general fiber 
of $a$ is a surface of negative Kodaira dimension and hence $a_*\omega_X=0$. 
The degeneration of the Leray spectral sequence associated to $a$
yields isomorphisms
\begin{gather*}
H^1(X,\omega_X\otimes a^*L)\cong H^0(C,R^1a_*\omega_X\otimes L)\\
H^2(X,\omega_X\otimes a^*L)\cong H^1(C,R^1a_*\omega_X\otimes L)
\oplus H^0(C,\omega_C\otimes L)
\end{gather*}
for every $L\in {\rm Pic}^0(C)$.
We distinguish two cases: $g(C)=q(X)=1$ and $g(C)=q(X)>1$.
If $g(C)=q(X)=1$, then $C={\rm Alb}(X)$. Moreover, by \cite{Ha} Corollary 4.2, $R^1a_*\omega_X$ is a $GV$-sheaf on ${\rm Alb}(X)$ and hence 
the thesis.
On the other hand, if $g(C)=q(X)>1$, then
$V^0(\omega_C)={\rm Pic}^0(C)$ and we conclude by invoking one more time Lemma \ref{p2}. 
\end{proof}
\begin{rmk}\label{q=1}
%
In the case $q(X)=1$, the previous propositions yield the following statement: 
for each $k$, $\ddim V^k(\omega_X)=1$ if and only if $\ddim V^k(\omega_Y)=1$.
In general, we have not been able to show that if a locus $V^k(\omega_X)$ is empty
(\emph{resp}. of dimension zero) then the corresponding locus $V^k(\omega_Y)$ is empty (\emph{resp}. of dimension zero).
This ambiguity is mainly caused by the possible presence of non-trivial automorphisms. 

An application of a sheafified version of the derivative complex (\emph{cf}. \cite{EL} Theorem 3 and \cite{LP})
can be shown to yield Conjecture \ref{intrP} for threefolds having $q(X)=2$ (see \cite{Lo}).
\end{rmk}

\section{Applications}
In this final section, we 
prove Corollaries \ref{intreuler}, \ref{intrh02} and \ref{intrfibra}. Moreover, we present a 
further result regarding
the invariance of the Euler characteristics for powers of the canonical bundle for derived equivalent smooth minimal varieties
of maximal Albanese dimension.
\subsection{Holomorphic Euler characteristic and Hodge numbers}
\begin{proof}[Proof of Corollary \ref{intreuler}.]
Let $d:=\ddim X$.
We begin with the case $\ddim {\rm alb}_X(X)=d$. By Theorem \ref{intralb},
$Y$ is of maximal Albanese dimension as well and by \eqref{foralb}
$$\codim V^1(\omega_X)\geq 1\quad \mbox{ and }\quad \codim V^1(\omega_Y)\geq 1.$$
We distinguish two cases: $V^0(\omega_X)\subsetneq {\rm Pic}^0(X)$ and $V^0(\omega_X)= {\rm Pic}^0(X)$.
If $V^0(\omega_X)\subsetneq {\rm Pic}^0(X)$, then $V^0(\omega_Y)\subsetneq {\rm Pic}^0(Y)$ by Proposition \ref{hzero} and hence 
$\chi (\omega_X)=\chi (\omega_Y)=0$.
On the other hand, if $V^0(\omega_X)={\rm Pic}^0(X)$, then by Proposition \ref{hzero} and \eqref{forinc}
$$\exists \, L\in V^0(\omega_X)\backslash \big(\cup_{k=1}^d V^k(\omega_X)\big)
\mbox{ such that } F(1,L)=(1,M)\mbox{ with } M\in V^0(\omega_Y)\backslash \big(\cup_{k=1}^d V^k(\omega_Y)\big).$$ 
Hence
$\chi(\omega_X)=\chi (\omega_X\otimes L)=
h^0(X,\omega_X\otimes L)=h^0(Y,\omega_Y\otimes M)=\chi (\omega_Y\otimes M)=\chi(\omega_Y)$.

We suppose now $\ddim {\rm alb}_X(X)=d-1$ and $\kappa(X)\geq 0$. By Theorem \ref{intralb}, we have $\ddim {\rm alb}_Y(Y)=d-1$ as well.
In the following we will deliberately use \eqref{forinc}.
We distinguish four cases.

The first case is when $V^0(\omega_X)=V^1(\omega_X)={\rm Pic}^0(X)$. 
By Proposition \ref{hzero} and Corollary \ref{V1}, $V^0(\omega_Y)=V^1(\omega_Y)={\rm Pic}^0(Y)$ as well.
We claim that 
$$\exists \,\, \sO_X\neq L\in V^0(\omega_X)\backslash V^2(\omega_X)\,\, \mbox{ such that } \,\, F(1,L)=(1,M)\,\mbox{ with }\, 
\sO_Y \neq M\in V^0(\omega_Y)\backslash V^2(\omega_Y).$$ 
In fact, by Remark \ref{rmksplit}, $F(1,{\rm Pic}^0(X))= (1,{\rm Pic}^0(Y))$ and
then it is enough to choose the preimage, under $F^{-1}$,
of an element $(1,M)$ with $M\notin V^2(\omega_Y)$. 
By using Corollary \ref{formula} twice, first with $k=0$ and hence with $k=1$, we obtain
\begin{gather*}
\chi(\omega_X)  = \chi (\omega_X\otimes L)  =  h^0(X,\omega_X\otimes L)-h^1(X,\omega_X\otimes L)=\\
  h^0(Y,\omega_Y\otimes M)-h^1(Y,\omega_Y\otimes M)
 =  \chi (\omega_Y\otimes M)=\chi (\omega_Y).
\end{gather*}

The second case is when $V^0(\omega_X)={\rm Pic}^0(X)$ and $V^1(\omega_X)\subsetneq {\rm Pic}^0(X)$. 
By Proposition \ref{hzero} and Corollary \ref{V1} we have $V^0(\omega_Y)={\rm Pic}^0(Y)$ and $V^1(\omega_Y)\subsetneq {\rm Pic}^0(Y)$.
As before, $F(1,{\rm Pic}^0(X))=(1,{\rm Pic}^0(Y))$, and hence we can pick an element 
$$\sO_X\neq L\in V^0(\omega_X)\backslash V^1(\omega_X)\; \mbox{ such that }\; F(1,L)=(1,M) \; \mbox{ with }\;
\sO_Y\neq M\in V^0(\omega_Y)\backslash 
V^1(\omega_Y).$$ Therefore
$\chi (\omega_X)=\chi(\omega_X\otimes L)=h^0(X,\omega_X\otimes M)=h^0(Y,\omega_Y\otimes M)=\chi (\omega_Y\otimes M)=\chi (\omega_Y)$.

The third case is when $V^0(\omega_X)\subsetneq {\rm Pic}^0(X)$ and $V^1(\omega_X)={\rm Pic}^0(X)$. 
Then $V^0(\omega_Y)\subsetneq {\rm Pic}^0(Y)$
and $V^1(\omega_Y)={\rm Pic}^0(Y)$ as well by Proposition \ref{hzero} and Corollary \ref{V1}.
Note that $F(1,{\rm Pic}^0(X))=(1,{\rm Pic}^0(Y))$ by Remark \ref{rmksplit}. 
Similarly to the previous cases, there exists a pair $(L,M)\neq(\sO_X,\sO_Y)$ such that $$F(1,L)=(1,M)\;\mbox{ with }\;
L\notin V^0(\omega_X)\cup V^2(\omega_X)\;\mbox{ and }\;
M\notin V^0(\omega_Y)\cup V^2(\omega_Y).$$ By Corollary \ref{formula}, we have
$\chi(\omega_X)=\chi (\omega_X\otimes L)=-h^1(X,\omega_X\otimes L)=-h^1(Y,\omega_Y\otimes M)=\chi (\omega_Y\otimes M)=\chi (\omega_Y)$.

The last case is when both $V^0(\omega_X)$ and $V^1(\omega_X)$ are proper subvarieties of ${\rm Pic}^0(X)$. 
Then $V^0(\omega_Y)$ and $V^1(\omega_Y)$ are proper subvarieties as well and $\chi(\omega_X)=\chi(\omega_Y)=0$.
\end{proof}
\begin{proof}[Proof of Corollary \ref{intrh02}]
 By the invariance of the Hochschild homology $HH_k(X)\cong HH_k(Y)$ for $k=0,1$, 
we have $h^0(X,\omega_X)=h^0(Y,\omega_Y)$ and $h^1(X,\omega_X)=h^1(Y,\omega_Y)$.
Then Corollary \ref{intreuler} implies $h^{0,2}(X)=h^{0,2}(Y)$.
The second equality follows at once by the isomorphism $HH_2(X)\cong HH_2(Y)$.
\end{proof}
%

Using a result in \cite{PP2}, we can also derive a consequence about pluricanonical bundles.
\begin{cor}\label{corpowers}
Let $X$ and $Y$ be smooth projective derived equivalent varieties.
Suppose that $X$ is minimal and of maximal Albanese dimension.
Then, $$\chi (\omega_X^{\otimes m})=\chi (\omega_Y^{\otimes m})\quad \mbox{ for all }\quad m\geq 2.$$
\end{cor}

\begin{proof}
By \cite{PP2} Corollary 5.5, $\omega_X^{\otimes m}$ and $\omega_Y^{\otimes m}$ are $GV$-sheaves 
on $X$ and $Y$ respectively for any $m\geq 2$.\footnote{The minimality condition is necessary; see \cite{PP2} Example 5.6.} 
Since $V^0(\omega_X^{\otimes m})\cong V^0(\omega_Y^{\otimes m})$, we argue as in the proof of Corollary \ref{intreuler}.
\end{proof}

\subsection{Fibrations}
In this subsection we study the behavior of particular types of fibrations under derived equivalence. We also prove Corollary \ref{intrfibra}. 
We begin by recalling some terminology from \cite{Cat} and \cite{LP}.
A smooth projective variety $X$ is of \emph{Albanese general type} 
if it is of maximal Albanese dimension and has non-surjective Albanese map.
An \emph{irregular fibration} (\emph{resp. higher irrational pencil}) is 
a surjective morphism with connected fibers $f:X\longrightarrow Z$ onto a normal variety $Z$ 
with $0<\ddim Z<\ddim X$ and such that any smooth model of $Z$
is of maximal Albanese dimension (\emph{resp}. Albanese general type).

In \cite{Po} Corollary 3.4, Popa observes that a consequence of Conjecture \ref{intrPV} is that if $X$ 
admits a fibration onto a variety having non-surjective Albanese map, then any Fourier-Mukai partner of $X$ admits
an irregular fibration. With Theorem \ref{intrV} at hand, 
we can verify this statement under an additional hypothesis on $X$.
\begin{prop}\label{fibration}
Let $X$ and $Y$ be smooth projective derived equivalent varieties with $\ddim {\rm alb_X}(X)\geq \ddim X-1$.
If $X$ admits a surjective morphism $f:X\longrightarrow Z$ with connected fibers onto a normal variety $Z$ 
having non-surjective Albanese map and such that $\ddim X>\ddim Z$, then $Y$ admits an irregular fibration. 
\end{prop}
\begin{proof}
 Let $Z\stackrel{f'}{\longrightarrow} Z'\longrightarrow {\rm alb}_Z(Z)$ be the Stein factorization of ${\rm alb}_Z$.
By taking a non-singular representative of $f'$ we can assume $Z'$ smooth.
We can easily check that $Z'$ is of maximal Albanese dimension and with non-surjective Albanese map.
 Hence $h^0(Z',\omega_{Z'})>0$ and thus, by \cite{EL} Proposition 2.2, there
 exists a positive-dimensional irreducible component $V$ of $V^0(\omega_{Z'})$ passing through the origin.
By Lemma \ref{p2},
$(f\circ f')^*V\subset V^k(\omega_X)_0$ where $k=\ddim X-\ddim Z'$. 
By \eqref{foralb}, $(f\circ f')^*V\subset V^k(\omega_{X})_0\subset V^1(\omega_X)_0$ and,
by Theorem \ref{intrV} (iii), there exists a positive-dimensional irreducible component $V'\subset V^1(\omega_Y)_0$. We conclude by applying 
\cite{GL2} Theorem 0.1.
\end{proof}

We point out that, thanks to Theorem \ref{intr3F}, we can remove the hypothesis 
``$\ddim {\rm alb}_X(X)\geq \ddim X-1$'' from the above proposition in the case of threefolds.
The following proposition, together with the subsequent remark, provides the proof of Corollary \ref{intrfibra}.

\begin{prop}\label{corfib}
 Let $X$ and $Y$ be smooth projective derived equivalent threefolds. Fix $k$ to be either 1 or 2.
Then $X$ admits a higher irrational pencil $f:X\longrightarrow Z$ with $0<\ddim Z\leq k$ if and only if 
$Y$ admits a higher irrational pencil $g:Y\longrightarrow W$ with $0<\ddim W\leq k$.
\end{prop}
\begin{proof}
Suppose first $k=1$. Let $f:X\longrightarrow Z$ be a higher irrational pencil
onto a smooth curve $Z$ of genus $g(Z)\geq 2$.
By Lemma \ref{p2}, $f^*V^0(\omega_Z)=f^*{\rm Pic}^0(Z)\subset V^2(\omega_X)_0$. By Proposition \ref{intr3F} (i), there exists
a component $T\subset V^2(\omega_Y)_0$ such that 
\begin{equation}\label{eqT}
\ddim T\geq q(Z)\geq 2. 
\end{equation}
By \cite{GL2} Theorem 0.1 or by \cite{Be2} Corollaire 2.3, there exists an
irrational fibration 
$g:Y\longrightarrow W$ onto a smooth curve $W$ such that $T\subset g^*{\rm Pic}^0(W)+\gamma$ for some $\gamma\in {\rm Pic}^0(Y)$. Therefore 
\begin{equation}\label{eqW}
q(W)=g(W)\geq \ddim T\geq 2
\end{equation}
and $g$ is in effect a higher irrational pencil. 

Suppose now $k=2$. Let $f:X\longrightarrow Z$ be a higher irrational pencil. It is a general fact that, 
by possibly replacing $Z$ with a lower dimensional variety, 
one can furthermore assume that $\chi(\omega_{Z'})>0$ for any smooth model $Z'$ of $Z$ (see \cite{PP1} p. 271). 
If $\ddim Z=1$, we apply the argument of the previous case. Suppose then $\ddim Z=2$. 
Then $q(Z)\geq 3$ and by Lemma \ref{p2}, 
$$f^*V^0(\omega_{Z})=f^*{\rm Pic}^0(Z)\subset V^1(\omega_X)_0.$$ By Theorem \ref{intr3F}, there exists
a component $T\subset V^1(\omega_Y)_0$ such that $\ddim T\geq q(Z')\geq 3$ and, by \cite{GL2} Theorem 0.1, 
there exists an irregular
fibration $g:Y\longrightarrow W$ such that $T\subset g^*{\rm Pic}^0(W)+\gamma$ for some $\gamma \in {\rm Pic}^0(Y)$. 
Then we conclude that $q(W)\geq \ddim T\geq 3$ and that 
$g$ is a higher irrational pencil.
\end{proof}
 \begin{rmk}\label{rmkfib}
We can slightly improve Proposition \ref{corfib} by keeping track of the irregularities of the fibrations.
By going back to the proof of Proposition \ref{corfib} for the case $k=1$, we see that by \eqref{eqT} and \eqref{eqW} 
we obtain $q(W)\geq q(Z)$.
Then we can formulate the following stronger statement. Fix an integer $g\geq 2$. 
The variety $X$ admits a higher irrational pencil $f:X\longrightarrow C$ onto a curve of genus $g(C)\geq g$ if and only if
$Y$ admits a higher irrational pencil $h:Y\longrightarrow D$ onto a curve of genus $g(D)\geq g$.
\end{rmk} 

\noindent \textbf{Acknowledgements.}
I am deeply grateful to Mihnea Popa and Christian Schnell for their insights, hints and encouragement. 
I also thank Chih-Chi Chou, Lawrence Ein, Victor Gonz\'{a}lez Alonso, Emanuele Macr\`{i}, Wenbo Niu and Tuan Pham for 
helpful conversations.

\end{document}